\documentclass[12pt]{article}
\usepackage{amsthm,amsmath,amssymb,latexsym,amscd}
\usepackage[all]{xy}
\newcommand{\g}{\frak{g}}

\newcommand{\s}{\frak{s}}

\newcommand{\x}{\mathbf{x}}
\newcommand{\y}{\mathbf{y}}
\newcommand{\z}{\mathbf{z}}

\newcommand{\A}{\mathcal{A}}

\newcommand{\M}{\mathcal{M}}

\renewcommand{\P}{\mathcal{P}}

\newcommand{\T}{\mathcal{T}}

\newcommand{\Hom}{\mathrm{Hom}}

\newcommand{\wt}{\widetilde}

\newcommand{\Ass}{\mathcal{A}ss}
\newcommand{\Com}{\mathcal{C}om}

\newcommand{\Lie}{\mathcal{L}ie}

\newcommand{\Zinb}{\mathcal{Z}inb}

\newcommand{\p}{\prime}

\newcommand{\pa}{\partial}
\renewcommand{\c}{\circ}
\newcommand{\ot}{\otimes}
\newcommand{\ti}{\tilde}

\newcommand{\<}{\langle}
\renewcommand{\>}{\rangle}

\newtheorem{definition}{Definition}[section]
\newtheorem{lemma}[definition]{Lemma}

\newtheorem{proposition}[definition]{Proposition}
\newtheorem{theorem}[definition]{Theorem}
\newtheorem{corollary}[definition]{Corollary}

\date{}

\begin{document}

\title{
Formal symplectic geometry for Leibniz algebras
}
\author{K. UCHINO}
\maketitle
\abstract{
We study a formal symplectic geometry
for anticyclic Leibniz operad and its Koszul dual operad.
}
\section{Introduction}

Let $\g$ be a finite dimensional Lie algebra.
In symplectic or Poisson geometry,
the Lie algebra structure on $\g$
is characterized as an odd Hamilton function,
$\theta_{Lie}$, over an even symplectic
plane, $\T^{*}\Pi\g=\Pi(\g\times\g^{*})$,
satisfying a Maurer-Cartan equation
$\{\theta_{Lie},\theta_{Lie}\}=0$
(Kosmann-Schwarzbach \cite{KosmannS},
see also Roytenberg \cite{Roy}).
Here $\{.,.\}$ is the canonical Poisson bracket
defined on the symplectic plane.
The Hamiltonian system $(\T^{*}\Pi\g,\theta_{Lie})$
defines a classical theory which should be quantized.
We consider a noncommutative version
of this Hamiltonian formalism.
The noncommutative Lie algebra is known as
a Leibniz algebra (Loday \cite{Loday}).
A Leibniz algebra is a vector space equipped with
a noncommutative binary bracket satisfying
the Leibniz identity
$$
[x_{1},[x_{2},x_{3}]]=[[x_{1},x_{2}],x_{3}]+
[x_{2},[x_{1},x_{3}]].
$$
The main aim of this note is to construct
a Hamiltonian system which characterizes
the finite dimensional Leibniz algebra.
Since the Leibniz algebra is noncommutative and
nonassociative, the ordinary manifold,
whether graded or not, is useless for our aim.
So we will use the theory of formal operad-geometry
introduced by Kontsevich \cite{Kont}.
According to Getzler-Kapranov \cite{GetKap},
the formal operad-geometry is a part of
cyclic (co)homology theory.
In general, if $\P$ is a cyclic binary quadratic operad,
then a cyclic (co)homology group is well-defined
in the category of $\P$-algebras.
Kontsevich proved in the cases of $\P=\Com,\Lie,\Ass$
that if $\A$ is a finite dimensional $\P$-algebra,
then the cyclic cohomology theory over $\A$
can be interpreted as a formal symplectic
geometry via the Koszul duality theory.\\
\indent
We consider the case of Leibniz operad.
It is known that the Leibniz operad is anti-cyclic,
although not cyclic (Chapoton \cite{Chap}).
Hence one can construct an anti-cyclic cohomology theory
in the category of Leibniz algebras (Uchino \cite{Uchi}).
We will prove that if $\g$ is a finite dimensional
Leibniz algebra, then the anti-cyclic cohomology theory
over $\g$ can be interpreted as a formal symplectic
geometry as with the cyclic case.\\
\indent
If $\P$ is a cyclic (resp. anti-cyclic)
binary quadratic operad,
then a finite generated
(co)free $\P^{!}$-(co)algebra is
a formal $\P^{!}$-manifold.
Here $\P^{!}$ is the Koszul dual of $\P$.
In the category of $\P^{!}$-manifolds
(so-called $\P^{!}$-world),
a ``Lie" algebra is a $\P$-algebra
and a formal function over a $\P^{!}$-manifold
is a cyclic (resp. anti-cyclic)
cochain in the category of $\P$-algebras.
Roughly speaking, a $\P^{!}$-manifold
is a cyclic (resp. anti-cyclic)
cohomology complex in the category of
$\P$-algebras.
The case of $\P^{!}=\Com$ and $\P=\Lie$
is the classical case above.\\
\indent
It is well-known that the Koszul dual operad
of the Leibniz operad is the Zinbiel operad (\cite{Loday}).
The quadratic relation of $\Zinb$ is
$$
x_{1}*(x_{2}*x_{3})=(x_{1}*x_{2})*x_{3}+(x_{2}*x_{1})*x_{3}.
$$
We call the Zinbiel world a Loday world.
In general, a formal function in $\P^{!}$-world
is expressed as the universal invariant bilinear form
defined on the free $\P^{!}$-algebra.
Hence our main problem is to give a tensor expression
of the universal invariant bilinear form on the
free Zinbiel algebra.
The tensor expression of the bilinear form
will be used to define
the canonical Poisson bracket in the Loday world.
We will see that the structure of a finite dimensional
Leibniz algebra is a formal function $\mu$
satisfying $\{\mu,\mu\}=0$,
where $\{.,.\}$ is a canonical Poisson bracket
in the Loday world.\\
\indent
As an application of the formal symplectic geometry,
we will study a metric tensor defined on a Leibniz algebra.
In terms of generalized geometry (Hitchin \cite{Hitchin}),
a Leibniz algebra is considered to be
a ``generalized Lie algebra".
It is known that a natural metric tensor $g(.,.)$
defined on a generalized Lie algebra
(=Leibniz algebra) satisfies
\begin{equation}\label{igsi}
g([x_{1},x_{2}],x_{3})+g(x_{2},[x_{1},x_{3}])
=g(x_{1},x_{2}\c x_{3}),
\end{equation}
where $[.,.]$ is a Leibniz bracket
and $x_{2}\c x_{3}:=[x_{2},x_{3}]+[x_{3},x_{2}]$.
In the classical world, a metric tensor is not
function on the symplectic plane $\T^{*}\Pi\g$,
because $C^{\infty}(\T^{*}\Pi\g)=\bigwedge(\g\oplus\g^{*})$.
On the other hand, in the Loday world,
a symmetric 2-tensor is a super function
on the formal symplectic plane.
This is an advantage that the Loday world has
over the classical one.
We will prove that (\ref{igsi}) is equivalent with
an invariant condition, that is, $\{\mu,g\}=0$.

\section{Leibniz and Zinbiel algebras}


A (left-)Leibniz algebra is a vector space $\g$
equipped with a binary bracket $[.,.]$
satisfying the Leibniz identity,
$$
[x_{1},[x_{2},x_{3}]]=[[x_{1},x_{2}],x_{3}]+
[x_{2},[x_{1},x_{3}]],
$$
where $x_{1},x_{2},x_{3}\in\g$.
A Zinbiel algebra
is a vector space equipped with
a binary product satisfying
$$
x_{1}*(x_{2}*x_{3})=(x_{1}*x_{2}+x_{2}*x_{1})*x_{3}.
$$
The operad of Zinbiel algebras is the Koszul dual
of the one of Leibniz algebras.
The Leibniz algebra and the Zinbiel algebra
are introduced and studied deeply
by Loday (\cite{Loday}).
Hence they are called Loday type algebras.\\
\indent
Let $V$ be a vector space.
The free Zinbiel algebra over $V$ is the
tensor space $\bar{T}V:=\bigoplus_{n\ge 1}V^{\ot n}$,
whose Zinbiel product is given by
$$
(...((x_{1}*x_{2})*x_{3})*\cdots)*x_{n}=
x_{1}\ot x_{2}\ot x_{3}\ot\cdots\ot x_{n}.
$$
For example,
$x_{1}*(x_{2}*x_{3})=
(x_{1}\ot x_{2}-x_{2}\ot x_{1})\ot x_{3}$.
By the universality of the free algebra,
for any Zinbiel algebra $(Z,*)$
and for any linear map $f:V\to Z$,
there exists a unique Zinbiel algebra morphism
$\hat{f}:\bar{T}V\to Z$ such that
the following diagram is commutative
$$
\xymatrix{
V\ar[r]^{\subset}\ar[dr]_{f} &
\bar{T}V\ar[d]^{\hat{f}}\\
& Z.
}
$$
\begin{lemma}
If $\x:=x_{1}\ot\cdots\ot x_{a}$
and $\y:=y_{1}\ot\cdots\ot y_{b}$, then
$$
sh(\x,\y)=\x*\y+\y*\x,
$$
where $sh(\x,\y)$
is the shuffle product of $\x$ and $\y$.
\end{lemma}
\indent
The cofree Zinbiel coalgebra over $\Pi\g$ is
the tensor space $\bar{T}^{c}\Pi\g=\bar{T}\Pi\g$,
whose coproduct is defined by
$$
\Delta(x_{1},...,x_{n}):=
\sum_{\substack{1\le i\le n-1 \\ \sigma}}
(-1)^{\sigma}(x_{\sigma(1)},...,x_{\sigma(i)})\ot
(x_{\sigma(i+1)},...,x_{\sigma(n-1)},x_{n}),
$$
where $\sigma$ is an $(i,n-1-i)$-unshuffle permutation, i.e.,
$\sigma(1)<\cdots<\sigma(i)$ and
$\sigma(i+1)<\cdots<\sigma(n-1)$.
Let $B:\bar{T}^{i}\Pi\g\to\Pi\g$ be an $i$-ary linear map
on $\Pi\g$.
The map is identified with a coderivation
on the coalgebra. The defining identity of
the coderivation is as follows. If $n\ge i$,
\begin{multline}\label{defcoderi}
B(x_{1},...,x_{n})=\sum_{j,\sigma}
(-1)^{\sigma}(-1)^{(i+1)j}\\
x_{\sigma(1)}\ot\cdots\ot x_{\sigma(j)}\ot
B(x_{\sigma(j+1)},...,x_{\sigma(j+i-1)},x_{i+j})\ot
x_{i+j+1}\ot\cdots\ot x_{n},
\end{multline}
where $\sigma$ is a $(j,i-1)$-unshuffle permutation
and the parity of $B$ is $i+1$.
The space of the coderivations becomes a Lie algebra
via the commutator.
Hence, the space of the multilinear maps,
$$
C_{Leib}^{\bullet}(\g):=\Hom(\bar{T}^{c}\Pi\g,\Pi\g),
$$
is also a Lie algebra. If $B:=[.,.]$ is binary
and if $B$ is a solution of $BB=0$,
then $(\g,B)$ becomes a Leibniz algebra and
$(C_{Leib}^{\bullet}\g,B)$
is the cohomology complex of Loday-Pirashvili \cite{LP}.
We call a cochain $B\in C_{Leib}^{\bullet}(\g)$
a bar-cochain.
\medskip\\
\indent
An invariant bilinear form in the category
of Leibniz algebras is an anti-symmetric 2-form,
$\<x_{1},x_{2}\>=-\<x_{2},x_{1}\>$, satisfying
\begin{eqnarray}
\label{invL1}
\<x_{1},[x_{2},x_{3}]\>&=&-\<[x_{2},x_{1}],x_{3}\>,\\
\label{invL2}
\<x_{1},[x_{2},x_{3}]\>&=&
\<[x_{1},x_{3}]+[x_{3},x_{1}],x_{2}\>.
\end{eqnarray}
Suppose that $\g$ is a finite dimensional Leibniz algebra.
Let $\g^{*}$ be the dual space of $\g$.
The coadjoint representation of $\g$ by $\g^{*}$
is defined by
\begin{eqnarray}\label{coadction}
\label{coadaction1}
\<x_{1},[x_{2},a]\>&=&-\<[x_{2},x_{1}],a\>,\\
\label{coadaction2}
\<x_{1},[a,x_{2}]\>&=&\<[x_{2},x_{1}]+[x_{2},x_{1}],a\>,
\end{eqnarray}
where $a\in\g^{*}$
and $\<.,.\>$ is the canonical pairing of
$\g$ and $\g^{*}$.
The double space $\g\oplus\g^{*}$ is a symplectic plane,
whose symplectic structure is defined by
\begin{equation}\label{defomega}
\omega(x_{1}+a_{1},x_{2}+a_{2}):=
\<x_{1},a_{2}\>-\<x_{2},a_{1}\>.
\end{equation}
The semi-direct product algebra $\g\ltimes\g^{*}$
is a Leibniz algebra satisfying
the invariant condition (\ref{invL1})-(\ref{invL2})
with respect to $\omega$.\\
\indent
An invariant bilinear form
in the category of Zinbiel algebras
is an anti-symmetric 2-form,
$\<x_{1},x_{2}\>=-\<x_{2},x_{1}\>$, satisfying
\begin{eqnarray}
\label{zinv1}\<x_{1}*x_{2},x_{3}\>&=&\<x_{3}*x_{2},x_{1}\>,\\
\label{zinv2}\oint\<x_{1}*x_{2},x_{3}\>&=&0,
\end{eqnarray}
where $\oint$ is the cyclic summation for $1,2,3$.
The defining relations of the invariant bilinear forms
were introduced by Chapton \cite{Chap}.

\section{Anticyclic Leibniz operad}

\subsection{anticyclic cochains}

By definition,
an anticyclic $n-1$-cochain over $\g$ is
an $n$-ary linear function on $\bar{T}\Pi\g$ such that
\begin{equation}\label{defaccochain}
A(x_{1},...,x_{n})
=\frac{1}{n}A\Big([x_{1},[x_{2},...,[x_{n-1},x_{n}]]]\Big),
\end{equation}
where $[.,.]$ is the free Lie bracket, or commutator,
over $\Pi\g$. For example,
$$
A(x_{1},x_{2},x_{3})=\frac{1}{3}\big(
A(x_{1},x_{2},x_{3})+A(x_{1},x_{3},x_{2})
-A(x_{2},x_{3},x_{1})-A(x_{3},x_{2},x_{1})\big),
$$
where $|x_{i}|=odd$ for each $i\in\{1,2,3\}$.
We sometimes call the anticyclic cochain
an ac cochain for short.
In \cite{Uchi} it was proved that
the set of anticyclic cochains becomes a subcomplex
of the cohomology complex of Leibniz algebra.\\
\indent
We prove that the space of anticyclic cochains
over a symplectic plane becomes an even Lie algebra.
\begin{lemma}
The free Lie algebra over $\Pi\g$ is stable
for the coderivations.
\end{lemma}
\begin{proof}
Let $B$ be an $i$-ary bar-cochain in $C_{Leib}^{i}(\g)$.
The case of $i=2$ was proved in \cite{Uchi}.
Hence we suppose that when the arity of $B$ is $i-1$,
the lemma holds.
We should compute $B[x_{1},...,x_{n}]$,
where $[x_{1},...,x_{n}]$
is the right-normalized Lie bracket
$[x_{1},...,x_{n}]:=[x_{1},[x_{2},...,x_{n}]]$.
When $n=i$, the lemma obviously holds.
So assume that $B[x_{1},...,x_{n-1}]$
is an element of the free Lie algebra,
where $n-1>i$.
We have
$$
[x_{1},...,x_{n}]=
x_{1}\ot [x_{2},...,x_{n}]
-(-1)^{n-1}[x_{2},...,x_{n}]\ot x_{1}.
$$
Applying $B$ to the first term,
$$
B(x_{1}\ot [x_{2},...,x_{n}])=
B_{x_{1}}[x_{2},...,x_{n}]
+(-1)^{|B|}x_{1}\ot B[x_{2},...,x_{n}],
$$
where $B_{x_{1}}:=B(x_{1},\cdot,\cdot,...,\cdot)$.
Since the arity of $B_{x_{1}}$ is $i-1$
and the length of $[x_{2},...,x_{n}]$ is $n-1$,
by assumption of induction
$B_{x_{1}}[x_{2},...,x_{n}]$
and $B[x_{2},...,x_{n}]$
are elements of the free Lie algebra.
Applying $B$ to the second term, we have
$$
B([x_{2},...,x_{n}]\ot x_{1})=
B[x_{2},...,x_{n}]\ot x_{1}+X,
$$
where $X$ is the term which has $B(,...,x_{1})$.
It is easy to prove that $X=0$.
Therefore, we obtain
\begin{equation}\label{defBonLie}
B[x_{1},...,x_{n}]=
B_{x_{1}}[x_{2},...,x_{n}]+(-1)^{|B|}
[x_{1},B[x_{2},...,x_{n}]],
\end{equation}
which implies the desired result.
\end{proof}
From (\ref{defBonLie}), we can know how
$B[x_{1},...,x_{n}]$ is computed.
For example, if $n=4$ and the arity of $B$ is $3$
($|B|=even$),
\begin{multline*}
B[x_{1},x_{2},x_{3},x_{4}]=\\
[B(x_{1},x_{2},x_{3}),x_{4}]+
[x_{3},B(x_{1},x_{2},x_{4})]-
[x_{2},B(x_{1},[x_{3},x_{4}])]+
[x_{1},B[x_{2},x_{3},x_{4}]],
\end{multline*}
where $|x_{i}|:=odd$.\\
\indent
Let $(\s,\omega)$ be a symplectic plane,
where $\omega$ is a symplectic structure on $\s$.
Let $A$ be an anticyclic $n-1$-cochain over $\s$,
which is an $n$-linear function on $\bar{T}\Pi\g$.
The ac cochain is identified with a bar cochain via
the symplectic structure,
\begin{equation}\label{defham}
A=(-1)^{|B|}\omega(B,-).
\end{equation}
Let $A_{1}$ be an ac $i$-cochain,
let $A_{2}$ an ac $j$-cochain
and let $B_{1},B_{2}$ the bar-cochains corresponding to
$A_{1},A_{2}$ respectively.
The parities of $B_{1}$ and $B_{2}$ are $i+1$
and $j+1$, respectively.
Define $\{A_{1},A_{2}\}$ by
\begin{equation}\label{defAA}
\{A_{1},A_{2}\}:=(-1)^{i+j}\omega([B_{1},B_{2}],-),
\end{equation}
which is an $i+j$-ary linear function.
From (\ref{defaccochain}), we have
\begin{lemma}\label{lemma999}
If $A$ is an anticyclic cochain, then
$$
A(x_{1},...,x_{n})
=-(-1)^{n-k+1}A(x_{1},...,[x_{k},...,x_{n}],x_{k-1}).
$$
where $|x_{i}|:=odd$.
\end{lemma}
For example, when $n=3$,
$$
A(x_{1},x_{2},x_{3})=-A([x_{2},x_{3}],x_{1})
=-A(x_{2},x_{3},x_{1})-A(x_{3},x_{2},x_{1}).
$$
\begin{proposition}
The cochain defined in (\ref{defAA})
is again anticyclic and the bracket
$\{A_{1},A_{2}\}$ is an even Lie bracket
on the space of anticyclic cochains.
\end{proposition}
\begin{proof}
For the sake of simplicity, we suppose that
the parity of variable is even. We have
$$
[B_{1},B_{2}]=\sum B_{1}(,...,B_{2}(,...,),...,)
-(-1)^{(i+1)(j+1)}B_{2}(,...,B_{1}(,...,),...,).
$$
Hence
\begin{multline*}
(-1)^{i+j}
\omega\big([B_{1},B_{2}](,...,),x_{n}\big)=\\
\sum(-1)^{j+1}A_{1}(,...,B_{2}(,...,),...,x_{n})
-(-1)^{i+j+(i+1)(j+1)}\omega
\big(B_{2}(,...,B_{1}(,...,),...,),x_{n}\big).
\end{multline*}
where we put $n:=i+j$.
By Lemma \ref{lemma999},
\begin{eqnarray*}
\omega\big(B_{2}(,...,B_{1}(,...,),...,),x_{n}\big)
&=&(-1)^{j+1}A_{2}(,...,B_{1}(,...,),x_{k},...,x_{n})\\
&=&-(-1)^{j+1}A_{2}(,...,[x_{k},...,x_{n}],B_{1}(,...,))\\
&=&-\omega\big(B_{2}(,...,[x_{k},...,x_{n}]),
B_{1}(,...,)\big)
\end{eqnarray*}
since $\omega$ is symmetric on $\Pi\g$
\begin{eqnarray*}
&=&-(-1)^{(i+1)(j+1)}\omega\big(B_{1}(,...,),
B_{2}(,...,[x_{k},...,x_{n}])\big)\\
&=&-(-1)^{(i+1)(j+1)+(i+1)}
A_{1}(,...,B_{2}(,...,[x_{k},...,x_{n}])).
\end{eqnarray*}
Hence we have
\begin{multline*}
(-1)^{i+j}
\omega\big([B_{1},B_{2}](,...,),x_{n}\big)=\\
\sum(-1)^{j+1}A_{1}(,...,B_{2}(,...,),...,x_{n})+
(-1)^{j+1}A_{1}\big(,...,B_{2}(,...,[x_{k},...,x_{n}])\big).
\end{multline*}
From (\ref{defaccochain}) and (\ref{defBonLie}),
one can see through that
\begin{equation}
(-1)^{i+j}\omega\big([B_{1},B_{2}](x_{1},...,x_{n-1}),x_{n}\big)=
(-1)^{j+1}\frac{1}{i}A_{1}B_{2}[x_{1},...,x_{n}].
\end{equation}
This implies that $\{A_{1},A_{2}\}$ is
an anticyclic cochain.
\end{proof}
We notice that $\{A_{1},A_{2}\}\sim\omega(B_{1},B_{2})$.
By a direct computation one can show that
\begin{proposition}\label{poissonbracket1}
\begin{equation}\label{keyid1}
\{A_{1},A_{2}\}(x_{1},...,x_{n})
=(-1)^{i+1}
\omega(B_{1},B_{2})(1^{\ot i}\ot T)[x_{1},...,x_{n}],
\end{equation}
where $T$ is the transposition of tensor,
$T(x_{1},...,x_{n}):=(\pm)(x_{n},...,x_{1})$.
\end{proposition}
For example, when $A_{1}$ is an ac 1-cochain
and $A_{2}$ is an ac 2-cochain,
\begin{eqnarray*}
(1\ot T)[x_{1},x_{2},x_{3}]
&=&x_{1}\ot T[x_{2},x_{3}]-x_{2}\ot T(x_{3}\ot x_{1})
-x_{3}\ot T(x_{2}\ot x_{1})\\
&=&-x_{1}\ot[x_{2},x_{3}]+x_{2}\ot(x_{1}\ot x_{3})
+x_{3}\ot(x_{1}\ot x_{2}),
\end{eqnarray*}
where we put $|x_{i}|=odd$ for each $i$.
Hence
\begin{multline*}
\{A_{1},A_{2}\}(x_{1},x_{2},x_{3})=\\
\omega\big(B_{1}(x_{1}),B_{2}[x_{2},x_{3}]\big)
-\omega\big(B_{1}(x_{2}),B_{2}(x_{1},x_{3})\big)
-\omega\big(B_{1}(x_{3}),B_{2}(x_{1},x_{2})\big),
\end{multline*}

\subsection{Universal invariant bilinear form}
Suppose that
$\g$ is a finite dimensional vector space.
Let $(p_{i})$ be a linear base of $\Pi\g$
and let $(q^{i})$ the dual base of $\Pi\g^{*}$.
Then the anticyclic cochain defined in (\ref{defaccochain})
is expressed as follows.
\begin{equation}\label{tensorexp}
A=\frac{1}{n}A_{i_{1},...,i_{n}}
[q^{i_{1}},...,q^{i_{n}}]_{*},
\end{equation}
where $[,...,]_{*}$ is the dual of
the normalized bracket $[,...,]$,
which is defined as follows.
\begin{equation}\label{defdualcommutator}
[x^{1},...,x^{n}]_{*}:=
x^{1}\ot[x^{2},...,x^{n}]_{*}-(-1)^{n-1}
x^{n}\ot[x^{1},...,x^{n-1}]_{*},
\end{equation}
where $|x_{i}|:=odd$ for each $i$.\\
\indent
In the following we suppose that
the parity of variables are even for
the sake of simplicity.
For any $x_{1},x_{2},x_{3}\in\g$,
the dual commutator satisfies
$[x_{1},x_{2}]_{*}=-[x_{1},x_{2}]_{*}$ and
\begin{eqnarray*}
[x_{1},x_{2},x_{3}]_{*}&=&[x_{3},x_{2},x_{1}]_{*},\\
\oint[x_{1},x_{2},x_{3}]_{*}&=&0,
\end{eqnarray*}
which are the same relations
as (\ref{zinv1}) and (\ref{zinv2}),
respectively.
Denote $\x:=x_{1}\ot\cdots\ot x_{a}$
and $\y:=y_{1}\ot\cdots\ot y_{b}$.
We put
\begin{equation}\label{univpairing}
\<\x,\y\>:=(-1)^{b+1}[\x,T\y]_{*},
\end{equation}
where $T\y$ is the transposition $\y$.
\begin{theorem}
This pairing is the universal invariant bilinear form
on the free Zinbiel algebra $\bar{T}\g$.
Namely, if $Z$ is a Zinbiel algebra equipped with
an invariant pairing $\<\cdot,\cdot\>^{\p}$
satisfying (\ref{zinv1})-(\ref{zinv2})
and if $f:\g\to Z$ is a linear map, then
the universal lift of $f$, $\hat{f}:\bar{T}\g\to Z$,
preserves the bilinear form.
\end{theorem}
First of all, we should check that
$\<\x,\y\>$ is antisymmetric.
\begin{lemma}
The dual commutator is triangular, i.e.,
$$
[x_{1},...,x_{n-1},x_{n}]_{*}=
(-1)^{n+1}[x_{n},x_{n-1},...,x_{1}]_{*}.
$$
\end{lemma}
\begin{proof}
When $n=2$, the identity holds.
By the assumption of induction,
\begin{eqnarray*}
[x_{1},...,x_{n}]_{*}&=&
x_{1}\ot[x_{2},...,x_{n}]_{*}-
x_{n}\ot[x_{1},...,x_{n-1}]_{*}\\
&=&(-1)^{n}
x_{1}\ot[x_{n},...,x_{2}]_{*}
-(-1)^{n}
x_{n}\ot[x_{n-1},...,x_{1}]_{*}\\
&=&(-1)^{n+1}[x_{n},x_{n-1},...,x_{1}]_{*}.
\end{eqnarray*}
\end{proof}
Thanks to the lemma above, we obtain
\begin{multline*}
\<\x,\y\>=(-1)^{b+1}[\x,T\y]_{*}
=(-1)^{b+1}(-1)^{a+b+1}[\y,T\x]_{*}
=-(-1)^{a+1}[\y,T\x]_{*}=\\
-\<\y,\x\>.
\end{multline*}
Secondly we prove that the pairing satisfies
(\ref{zinv1}).
\begin{lemma}
$sh(\x,\y)=x_{1}\ot sh(\x_{2},\y)+y_{1}\ot sh(\x,\y_{2})$,
where $\x_{2}:=x_{2}\ot\cdots\ot x_{n}$
and $\y_{2}$ is the same.
\end{lemma}
Denote $\z:=z_{1}\ot\cdots\ot z_{c}$.
From the axiom of Zinbiel algebra,
$$
\x*\y=(\x*\y^{b-1}+\y^{b-1}*\x)\ot y_{b}
=sh(\x,\y^{b-1})\ot y_{b},
$$
where $\y^{b-1}:=y_{1}\ot\cdots\ot y_{b-1}$.
Hence
$\<\x*\y,\z\>=\<sh(\x,\y^{b-1})\ot y_{b},\z\>$.
We should prove
\begin{equation}\label{lemmaid}
(-1)^{c+1}[sh(\x,\y^{b-1}),y_{b},T\z]_{*}
=(-1)^{a+1}[sh(\z,\y^{b-1}),y_{b},T\x]_{*}.
\end{equation}
From Lemma above,
\begin{multline}\label{noproof01}
(-1)^{c+1}[sh(\x,\y^{b-1}),y_{b},T\z]_{*}=\\
(-1)^{c+1}\Big([x_{1},sh(\x_{2},\y^{b-1}),y_{b},T\z]_{*}
+[y_{1},sh(\x,\y_{2}^{b-1}),y_{b},T\z]_{*}\Big),
\end{multline}
where $\y_{2}^{b-1}:=y_{2}\ot\cdots\ot y_{b-1}$.
The first term of (\ref{noproof01}) is
\begin{multline*}
(-1)^{c+1}[x_{1},sh(\x_{2},\y^{b-1}),y_{b},T\z]_{*}=\\
(-1)^{c+1}x_{1}\ot[sh(\x_{2},\y^{b-1}),y_{b},T\z]_{*}
+(-1)^{c}z_{1}\ot[x_{1},sh(\x_{2},\y^{b-1}),y_{b},T\z_{2}]_{*}=
\end{multline*}
by the assumption of induction
\begin{equation}\label{noproof02}
=(-1)^{a}x_{1}
\ot[sh(\z,\y^{b-1}),y_{b},T\x_{2}]_{*}+(-1)^{c}z_{1}
\ot[x_{1},sh(\x_{2},\y^{b-1}),y_{b},T\z_{2}]_{*}.
\end{equation}
The second term of (\ref{noproof01}) is in the same way
\begin{multline}\label{noproof03}
(-1)^{c+1}[y_{1},sh(\x,\y_{2}^{b-1}),y_{b},T\z]_{*}=\\
(-1)^{c+1}y_{1}\ot[sh(\x,\y_{2}^{b-1}),y_{b},T\z]_{*}
+(-1)^{c}z_{1}
\ot[y_{1},sh(\x,\y_{2}^{b-1}),y_{b},T\z_{2}]_{*}=\\
(-1)^{a+1}y_{1}\ot[sh(\z,\y_{2}^{b-1}),y_{b},T\x]_{*}+
(-1)^{c}z_{1}
\ot[y_{1},sh(\x,\y_{2}^{b-1}),y_{b},T\z_{2}]_{*}.
\end{multline}
(\ref{noproof02})+(\ref{noproof03}) is
\begin{multline*}
(-1)^{a}x_{1}\ot[sh(\z,\y^{b-1}),y_{b},T\x_{2}]_{*}
+(-1)^{c}z_{1}\ot[x_{1},sh(\x_{2},\y^{b-1}),y_{b},T\z_{2}]_{*}+\\
(-1)^{a+1}y_{1}\ot[sh(\z,\y_{2}^{b-1}),y_{b},T\x]_{*}
+(-1)^{c}z_{1}
\ot[y_{1},sh(\x,\y_{2}^{b-1}),y_{b},T\z_{2}]_{*}=\\
(-1)^{a}x_{1}\ot[sh(\z,\y^{b-1}),y_{b},T\x_{2}]_{*}+
(-1)^{a+1}y_{1}\ot[sh(\z,\y_{2}^{b-1}),y_{b},T\x]_{*}+\\
(-1)^{c}z_{1}\ot[sh(\x,\y^{b-1}),y_{b},T\z_{2}]_{*}=
\end{multline*}
by the assumption of induction again
\begin{multline*}
=(-1)^{a}x_{1}\ot[sh(\z,\y^{b-1}),y_{b},T\x_{2}]_{*}+
(-1)^{a+1}y_{1}\ot[sh(\z,\y_{2}^{b-1}),y_{b},T\x]_{*}+\\
(-1)^{a+1}z_{1}\ot[sh(\z_{2},\y^{b-1}),y_{b},T\x]_{*},
\end{multline*}
which is equal to the right-hand side of (\ref{lemmaid}).
Therefore,
$$
\<\x*\y,\z\>=\<\z*\y,\x\>.
$$
In the same way by using induction one can show that
\begin{equation}\label{univcyclic}
\oint\<\x*\y,\z\>=0.
\end{equation}
\indent
Finally we prove that
the pairing $\<\x,\y\>=(-1)^{b+1}[\x,T\y]_{*}$
is universal.
Let $Z$ be a Zinbiel algebra equipped with
an invariant bilinear form $\<.,.\>^{\p}$
and let $f:\g\to Z$ be a linear map.
We should prove that the lift
$\hat{f}$ preserves the pairing.
It suffices to consider
the case of $\<x_{1}\ot\cdots\ot x_{n-1},x_{n}\>$.
By (\ref{zinv1})-(\ref{zinv2}),
\begin{eqnarray*}
\<x_{1}\ot\cdots\ot x_{n-1},x_{n}\>
&=&\<x_{n}\ot x_{n-1},x_{1}\ot\cdots\ot x_{n-2}\>\\
&=&-\<x_{1}\ot\cdots\ot x_{n-2},x_{n}\ot x_{n-1}\>\\
&=&\cdots\\
&=&
(-1)^{n-1-i}\<x_{1}\ot\cdots\ot x_{i},x_{n}\ot\cdots\ot x_{i+1}\>\\
&=&(-1)^{n-1-i}\<\x^{i-1}\ot x_{i},T\x_{i+1}\>\\
&=&-(-1)^{n-1-i}\<T\x_{i+1}*\x^{i-1},x_{i}\>-(-1)^{n-1-i}\<x_{i}*T\x_{i+1},\x^{i-1}\>\\
&=&
-(-1)^{n-1-i}\<T\x_{i+1}*\x^{i-1},x_{i}\>-(-1)^{n-1-i}\<\x^{i-1}*T\x_{i+1},x_{i}\>\\
&=&(-1)^{n-i}\<sh(T\x_{i+1},\x^{i-1}),x_{i}\>.
\end{eqnarray*}
On the other hand, one can show that
\begin{equation*}\label{noproof04}
[x_{1},...,x_{n}]_{*}=
\sum_{i=1}^{n}(-1)^{n-i}
sh(T\x_{i+1},\x^{i-1})\ot x_{i},
\end{equation*}
where $sh(\emptyset,-)=sh(-,\emptyset)=id$.
We put
$$
(\hat{f}\ot f)\big(sh(T\x_{i+1},\x^{i-1})\ot x_{i}\big)
:=\<\hat{f}T\x_{i+1}*\hat{f}\x^{i-1}+
\hat{f}\x^{i-1}*\hat{f}T\x_{i+1},fx_{i}\>^{\p}.
$$
Then we obtain
$$
\frac{1}{n}(\hat{f}\ot f)
[x_{1},...,x_{n}]_{*}=\<\hat{f}\x,fx_{n}\>^{\p}.
$$
This means that $\<\x,\y\>=(-1)^{b+1}[\x,T\y]_{*}$
is the universal invariant bilinear form.

\section{Loday world}

Let $\g$ be a finite dimensional vector space
(not necessarily Leibniz algebra)
and let $F\g$ the space of anticyclic
cochains over $\g$.
Here $F\g=\bigoplus_{i\ge 2}F^{i}\g$
and $F^{i}\g$ is the space of ac $i-1$-cochains.
\begin{definition}
The triple
$\Pi\M:=(\bar{T}^{c}\Pi\g,\bar{T}\Pi\g^{*},F\g)$
is called a formal super Zinbiel manifold
or super Loday manifold,
where $\bar{T}^{c}\Pi\g$ and $\bar{T}\Pi\g^{*}$
are the cofree Zinbiel coalgebra over $\Pi\g$
and the free Zinbiel algebra over the dual space,
respectively.
\end{definition}
\begin{itemize}
\item
By definition a function or formal function
over the manifold is
an anticyclic cochain in $F\g$.
\item
A local coordinate of $\Pi\M$
is by definition a linear base of $\Pi\g^{*}$.
When $\g$ is an ordinary vector space,
the coordinate degree is odd.
\item
A vector field on $\Pi\M$ is by definition
a bar-cochain or equivalently coderivation
on the cofree coalgebra $\bar{T}^{c}\Pi\g$.
\end{itemize}
\indent
The above definition holds for any
binary quadratic cyclic or anticyclic operads.
We here give a general definition of
formal super operad-manifold.
Let $\P$ be a binary quadratic cyclic
(resp. anticyclic) operad,
let $\P^{!}$ the Koszul dual of $\P$ and
let $V$ a finite dimensional vector space.
The formal super $\P^{!}$-manifold over $V$
is the following data:\\
-- $\wt{\P}^{!c}\Pi{V}$ :
the cofree $\P^{!}$-coalgebra over $\Pi{V}$.\\
-- $\wt{\P}^{!}\Pi{V}$ :
the free $\P^{!}$-algebra over $\Pi{V}$.\\
-- $F(V,\P)$ : the space of cyclic
(resp. anticyclic) cochains over $V$ in the category
of $\P$-algebras.\\
The case of cyclic operad
was studied in \cite{Kont},
in particular when $\P=\Com,\Ass,\Lie$.\\
\indent
Let $\Pi\M$ be the super Loday manifold over $\g$
and let $(p_{i})$, $(q^{j})$ are linear bases
of $\Pi\g$ and $\Pi\g^{*}$ respectively.
The base $(q^{j})$ is a local coordinate of
the manifold.
\begin{definition}\label{defderivation}
The coordinate derivation of the function on $\Pi\M$
is defined as follows.
$$
\frac{\pa}{\pa q^{i}}
[x^{1},...,x^{n}]_{*}:=(\pm)\sum_{j=1}^{n}
x^{\sigma_{1}}\ot\cdots\ot x^{\sigma_{n-1}}
\ot\frac{\pa x^{\sigma(n)}}{\pa q^{i}}.
$$
Namely, after expansion,
the most right-component is derived.
\end{definition}
The derivation is a map of $F\g$
to the free Zinbiel algebra $\bar{T}\Pi\g$.\\
\indent
Consider the symplectic plane $\s:=\g\oplus\g^{*}$
with the symplectic structure $\omega$
defined in (\ref{defomega}).
\begin{definition}[cotangent bundle]
$\T^{*}\Pi\M:=
\big(\bar{T}^{c}\Pi\s,\bar{T}\Pi\s^{*},F\s\big)$.
\end{definition}
The canonical Poisson bracket over the cotangent bundle
is defined as follows.
\begin{definition}[Poisson bracket]
$$
\{A_{1},A_{2}\}:=
\sum_{i}(-1)^{|A_{1}|}\Big\<\frac{\pa A_{1}}{\pa p_{i}},\frac{\pa A_{2}}{\pa q^{i}}\Big\>-(-1)^{|A_{1}|}
\Big\<\frac{\pa A_{1}}{\pa q^{i}},\frac{\pa A_{2}}{\pa p_{i}}\Big\>,
$$
where $A_{1},A_{2}\in F\s$
and $\<.,.\>$ is the universal invariant bilinear form
introduced in Section 3.2.
\end{definition}
$\frac{\pa A}{\pa p_{i}}$ and $\frac{\pa A}{\pa q_{i}}$
are respectively equal to
$(\pm)A(,...,q^{i})$ and
$(\pm)A(,...,p^{i})$.
This implies that
the Poisson bracket is equivalent with
the graded Lie bracket
in Proposition \ref{poissonbracket1}.
\begin{definition}[Hamiltonian vector field]
Let $A$ be a function over $\T^{*}\Pi\M$
or anticyclic cochain over $\s$.
The coderivation $B$ defined by (\ref{defham})
is called a Hamiltonian vector field of $A$.
\end{definition}

\begin{definition}[structures]
A function, $\theta$, over $\T^{*}\Pi\M$ is called
a structure, if it is a cubic form satisfying
$\{\theta,\theta\}=0$.
A $Q$-structure is the Hamiltonian vector field
of $\theta$.
\end{definition}

Let $[.,.]$ be a binary bracket product on $\g$,
which can be extended on $\s$ via the coadjoint action
(\ref{coadaction1})-(\ref{coadaction2}).
We put
$$
\mu:=C_{ij}^{k}[q^{i},q^{j},p_{k}]_{*},
$$
where
$C_{ij}^{k}:=\omega([p_{i},p_{j}],q_{k})$.
\begin{theorem}
$\{\mu,\mu\}=0$ if and only if
$[.,.]$ is a Leibniz bracket.
\end{theorem}
\begin{proof}
The proof is by a direct computation.
We denote $x\ot y$ by shortly $xy$.
Then
$$
\mu=C_{ij}^{k}q^{i}q^{j}p_{k}+C_{ij}^{k}q^{i}p_{k}q^{j}
-C_{ij}^{k}p_{k}q^{i}q^{j}-C_{ij}^{k}p_{k}q^{j}q^{i}.
$$
It suffices to compute
$\<\frac{\pa\mu}{\pa p_{a}},\frac{\pa\mu}{\pa q^{a}}\>$.
By the definition of the derivation,
\begin{eqnarray*}
\frac{\pa\mu}{\pa p_{a}}&=&C_{ij}^{a}q^{i}q^{j}\\
\frac{\pa\mu}{\pa q^{a}}&=&C_{ia}^{k}q^{i}p_{k}
-C_{ia}^{k}p_{k}q^{i}-C_{aj}^{k}p_{k}q^{j}\\
&=&
C_{ka}^{l}q^{k}p_{l}
-C_{ka}^{l}p_{l}q^{k}
-C_{ak}^{l}p_{l}q^{k}
\end{eqnarray*}
The first pairing is
$\<C_{ij}^{a}q^{i}q^{j},C_{ka}^{l}q^{k}p_{l}\>
=C_{ij}^{a}C_{ka}^{l}\<q^{i}q^{j},q^{k}p_{l}\>$.
By the invariant condition,
\begin{eqnarray*}
\<q^{i}q^{j},q^{k}p_{l}\>&=&-\<q^{k}p_{l},q^{i}q^{j}\>\\
&=&\<(q^{i}q^{j})q^{k},p_{l}\>-\<p_{l}(q^{i}q^{j}),q^{k}\>\\
&=&\<q^{i}q^{j}q^{k},p_{l}\>+\<q^{k}(q^{i}q^{j}),p_{l}\>\\
&=&\<q^{i}q^{j}q^{k},p_{l}\>+\<q^{k}q^{i}q^{j},p_{l}\>
-\<q^{i}q^{k}q^{j},p_{l}\>.
\end{eqnarray*}
Hence
\begin{eqnarray*}
\<C_{ij}^{a}q^{i}q^{j},C_{ka}^{l}q^{k}p_{l}\>
&=&
C_{ij}^{a}C_{ka}^{l}\<q^{i}q^{j}q^{k},p_{l}\>+
C_{ij}^{a}C_{ka}^{l}\<q^{k}q^{i}q^{j},p_{l}\>
-C_{ij}^{a}C_{ka}^{l}\<q^{i}q^{k}q^{j},p_{l}\>\\
&=&
C_{ij}^{a}C_{ka}^{l}\<q^{i}q^{j}q^{k},p_{l}\>+
C_{jk}^{a}C_{ia}^{l}\<q^{i}q^{j}q^{k},p_{l}\>
-C_{ik}^{a}C_{ja}^{l}\<q^{i}q^{j}q^{k},p_{l}\>.
\end{eqnarray*}
In the same way,
\begin{eqnarray*}
\<C_{ij}^{a}q^{i}q^{j},-C_{ka}^{l}p_{l}q^{k}\>
&=&-C_{ij}^{a}C_{ka}^{l}\<q^{i}q^{j},p_{l}q^{k}\>\\
&=&C_{ij}^{a}C_{ka}^{l}\<p_{l}q^{k},q^{i}q^{j}\>\\
&=&-C_{ij}^{a}C_{ka}^{l}\<q^{i}q^{j}q^{k},p_{l}\>
\end{eqnarray*}
and
\begin{eqnarray*}
\<C_{ij}^{a}q^{i}q^{j},-C_{ak}^{l}p_{l}q^{k}\>
=-C_{ij}^{a}C_{ak}^{l}\<q^{i}q^{j}q^{k},p_{l}\>.
\end{eqnarray*}
We obtain
\begin{eqnarray*}
\<\frac{\pa\mu}{\pa p_{a}},\frac{\pa\mu}{\pa q^{a}}\>
&=&C_{jk}^{a}C_{ia}^{l}\<q^{i}q^{j}q^{k},p_{l}\>
-C_{ik}^{a}C_{ja}^{l}\<q^{i}q^{j}q^{k},p_{l}\>
-C_{ij}^{a}C_{ak}^{l}\<q^{i}q^{j}q^{k},p_{l}\>\\
&=&(C_{jk}^{a}C_{ia}^{l}
-C_{ik}^{a}C_{ja}^{l}
-C_{ij}^{a}C_{ak}^{l})
\<q^{i}q^{j}q^{k},p_{l}\>\\
&=&([i,[j,k]]-[j,[i,k]]-[[i,j],k])\<q^{i}q^{j}q^{k},p_{l}\>.
\end{eqnarray*}
Therefore, if $[.,.]$ is a Leibniz bracket,
then $\{\mu,\mu\}=0$.
We put
$$
L_{ijk}^{l}:=C_{jk}^{a}C_{ia}^{l}
-C_{ik}^{a}C_{ja}^{l}
-C_{ij}^{a}C_{ak}^{l}
$$
By the definition of the pairing,
\begin{eqnarray*}
\<q^{i}q^{j}q^{k},p_{l}\>&=&[q^{i},q^{j},q^{k},p_{l}]_{*}\\
&=&q^{i}\ot[q^{j},q^{k},p_{l}]_{*}
+p_{l}\ot[q^{i},q^{j},q^{k}]_{*}\\
&\cdots&\cdots\\
&=&q^{i}\ot q^{j}\ot q^{k}\ot p_{l}+\cdots
\end{eqnarray*}
If $\{\mu,\mu\}=0$, then
$$
L_{ijk}^{l}q^{i}\ot q^{j}\ot q^{k}\ot p_{l}=0,
$$
which implies that $[.,.]$ is Leibniz.
\end{proof}
The function $\mu$ is a structure which
characterizes the semi-direct product
Leibniz algebra $\g\ltimes\g^{*}$.
More generally, when $\g\oplus\g^{*}$ is an Abelian
extension of $\g$ by $\g^{*}$,
the structure has the following form,
$$
\theta_{Leib}:=
C_{ij}^{k}[q^{i},q^{j},p_{k}]_{*}
+\frac{1}{3}H_{ijk}[q^{i},q^{j},q^{k}]_{*}
$$
and $\{\theta_{Leib},\theta_{Leib}\}=0$ if and only if
the twisted bracket
$$
[x_{1}+a_{1},x_{2}+a_{2}]
=[x_{1},x_{2}]+[x_{1},a_{2}]+[a_{1},x_{2}]+H(x_{1},x_{2})
$$
is a Leibniz bracket,
where $H:=\frac{1}{3}H_{ijk}[q^{i},q^{j},q^{k}]_{*}$.
\begin{definition}
Let $\mu$ be the structure defined above.
We put $b_{\mu}:=\{\mu,-\}$.
This becomes a coboundary operator on $F\g$.
The pair $(F\g,b_{\mu})$ is an anticyclic cohomology
complex over $\g$.
\end{definition}
Finally we study a metric tensor on $\g$.
An invariant bilinear form in the category of
Lie algebras is a symmetric tensor $g(.,.)$
satisfying the well-known condition,
$$
g(x_{1},[x_{2},x_{3}])=g([x_{1},x_{2}],x_{3}),
$$
where $[.,.]$ is an ordinary Lie bracket.
\begin{definition}
Let $g(.,.)$ be a symmetric bilinear form on $\g$.
We call $g$ a generalized symmetric invariant
bilinear form, if
\begin{equation}\label{gsi}
g([x_{1},x_{2}],x_{3})+g(x_{2},[x_{1},x_{3}])
=g(x_{1},x_{2}\c x_{3}),
\end{equation}
where $x_{2}\c x_{3}:=[x_{2},x_{3}]+[x_{3},x_{2}]$.
\end{definition}
If $\g$ is a Lie algebra as a commutative Leibniz algebra,
then (\ref{gsi}) is equal to the classical invariant
condition above. In general, a symmetric bilinear form
on $\g$ is a function over the cotangent bundle,
$$
g=\frac{1}{2}g_{ij}[q^{i},q^{j}]_{*}.
$$
The bilinear form $g$ is identified with
a linear map $\ti{g}:\g\to\g^{*}$
and satisfies (\ref{gsi})
if and only if the graph of $\ti{g}$ is a subalgebra
of the semi-direct product Leibniz algebra $\g\ltimes\g^{*}$.
\begin{corollary}
$g$ satisfies (\ref{gsi}) if and only if
$b_{\mu}g=\{\mu,g\}=0$.
\end{corollary}
\indent
Suppose that $g$ is nondegenerate
(i.e. pseudo-Euclidean metric).
The inverse $g^{-1}$ is also a function over $\T^{*}\Pi\M$,
$$
g^{-1}=\frac{1}{2}g^{ij}[p_{i},p_{j}]_{*}.
$$
We denote by $X_{g^{-1}}$ the Hamiltonian vector field
of $g^{-1}$.
The canonical transformation of $\mu$ by
the Hamiltonian flow $exp(X_{g^{-1}})$ is computed as follows.
$$
exp(X_{g^{-1}})(\mu)=\mu+\{\mu,g^{-1}\}
+\frac{1}{2}\{\{\mu,g^{-1}\},g^{-1}\}.
$$
If $g$ satisfies (\ref{gsi}), then
$exp(X_{g^{-1}})(\mu)=\mu+\{\mu,g^{-1}\}$,
and vice versa.
In that case,
$\nu:=\{\mu,g^{-1}\}$ is the second structure
and $\mu+\nu$ defines a Drinfeld double
in the Loday world.

\begin{verbatim}
Kyousuke UCHINO
email:kuchinon@gmail.com
\end{verbatim}

\begin{thebibliography}{}
\bibitem{Chap}
F. Chapoton.
{\em On some anticyclic operads}.
Algebr. Geom. Topol. 5 (2005), 53--69.
\bibitem{Hitchin}
N. Hitchin.
{\em Brackets, forms and invariant functionals}.
arXiv:math/0508618.
\bibitem{GetKap}
E. Getzler and M. Kapranov.
{\em Cyclic operads and cyclic homology}.
Geometry, topology, physics,
Conf. Proc. Lecture Notes Geom.
Topology, IV, Int. (1995), 167--201.
\bibitem{GK}
V. Ginzburg and M. Kapranov.
{\em Koszul duality for operads}.
Duke Math. J. 76 (1994), no. 1, 203--272.
{\em Erratum to: ``Koszul duality for operads''}.
Duke Math. J. 80 (1995), no. 1, 293.
\bibitem{Kont}
M. Kontsevich.
{\em Formal (Non)-Commutative Symplectic Geometry}.
The Gelfand Mathematical Seminars,
1990-1992 (1993), 173--187.
\bibitem{KosmannS}
Y. Kosmann-Schwarzbach.
{\em Jacobi quasi-bialgebras and quasi-Poisson Lie groups}.
Contemporary Mathematics. 132.
(1992) 459--489.
\bibitem{LP}
J-L. Loday and T. Pirashvili.
{\em Universal enveloping algebras of Leibniz algebras and (co)homology}.
Math. Ann. 296 (1993), no. 1, 139--158.
\bibitem{Loday}
J-L. Loday.
{\em Dialgebras}.
Lecture Notes in Mathematics 1763.
\bibitem{Roy}
D. Roytenberg.
{\em AKSZ-BV formalism and Courant algebroid-induced topological field theories}.
Lett. Math. Phys. 79(2) (2007), 143--159.
arXiv:hep-th/0608150
\bibitem{Uchi}
K. Uchino.
{\em
Derived bracket construction and anticyclic subcomplex
of Leibniz (co)homology complex}.
arXiv:1312.7268v1

\end{thebibliography}
\end{document}